\documentclass[10pt, a4paper]{article}

\usepackage{epsfig}
\usepackage{cite}
\usepackage{amsfonts}
\usepackage{amsmath}
\usepackage{amssymb}
\usepackage{mathtools}
\usepackage{amsxtra}
\usepackage{color}
\usepackage{rotating}

\usepackage{float}
\floatstyle{ruled}
\newfloat{algorithm}{thp}{lop}
\floatname{algorithm}{Table}
\usepackage{datetime}
\newdateformat{mydate}{\monthname[\THEMONTH] \THEYEAR}

\newcommand{\trace}[1]{\ensuremath{\operatorname{trace}(#1)} }
\newcommand{\rank}[1]{\ensuremath{\operatorname{rank}}(#1) }
\newcommand{\R}{\ensuremath{\mathbb{R}} }

 \newtheorem{coro}{\textbf{Corollary}}
 \newtheorem{lem}{\textbf{Lemma}}
 \newtheorem{thm}{\textbf{Theorem}}

\newenvironment{proof}{\paragraph{Proof:}}{\hfill\ensuremath{\square}}

\title{A novel representation of rank constraints for
  non-square real matrices}

\author{Ram\'on A. Delgado, Juan C. Ag\"uero and Graham C. Goodwin\\ % <-this % stops a space
School of Electrical Engineering and Computer Science,\\
 The University of Newcastle, Australia} 

\date{\mydate\today}
\begin{document}

\maketitle 

\begin{abstract}
We present a novel representation of rank
constraints for non-square real matrices. We establish
relationships with some existing results, which are particular
cases of our representation. One of these
particular cases, is a representation of  the $\ell_0$ pseudo-norm,
which is used in sparse
representation problems. Finally, we describe how our representation
can be included  in
rank-constrained optimization and in rank-minimization problems. 
\end{abstract}

\textbf{Notation and basic definitions:}  $\rank{A}$ denotes the rank of a
matrix $A$. We denote by  
   $A^\dagger$ to the Moore-Penrose pseudoinverse of $A$. $\lambda_i(A)$ denotes the i-th largest eigenvalue of a
symmetric matrix $A$,  $A \succeq 0$ denotes that $A$  is positive
semidefinite, and $A \succeq B$ denotes that $A -B \succeq 0$. We
represent the transpose of a given matrix $A$ as
$A^\top$. $\mathbb{S}^n$ denotes the set of symmetric matrices of size $n\times
n$, and $\mathbb{S}^n_+$ the set of symmetric positive semidefinite
matrices,
i.e. $\mathbb{S}^n_+\coloneqq\{A\in\mathbb{S}^n|A\succeq0\}$. $\|\cdot\|_F$ denotes the
Frobenius norm. $\|\cdot\|_0$ denotes the $\ell_0$ pseudo-norm that
counts the number of nonzero elements of a vector.
\section{Main Result}

\begin{thm}
\label{RC:thm:WG0_gen}
Let  $G\in\mathbb{R}^{m\times n}$, then the following expressions are equivalent
\begin{description}
\item[(i)] $\rank{G}\leq r$
\item[(ii)] $\exists W_R\in\Phi_{n,r}$, such that
  $GW_R=0_{m\times n}$
\item[(iii)] $\exists  W_L\in\Phi_{m,r}$, such that $W_LG=0_{m\times n}$
\end{description}
where
\begin{align} 
 \Phi_{n,r}=\{W\in\mathbb{S}^n,\;0\preceq W\preceq I, \trace{W}=n-r\}
\end{align}  
\end{thm}
\begin{proof}
Here we provide a sketch of the proof. A more detailed proof will be
published somewhere else.

 We first prove \textbf{(i)}$\implies$\textbf{(ii)}. Let  $\rank{G}\leq
r$ then there exists at least $n-r$ linearly independent vectors
$u_i\in\R^n$ such that $Gu_i=0$. Define
$U=[u_1,\dots,u_{n-r}]\in\R^{n\times (n-r)}$ having full column
rank. Then we can construct a orthogonal projector, $W_R=UU^\dagger$
which satisfies the condition $\rank{W_R}=n-r$ and is such that $GW_R=0$ . Since $W_R$ is
an orthogonal projector it also satisfies  $W_R\in\mathbb{S}^n$, $0\preceq
W_R\preceq  I_n$ and $\rank{W_R}=\trace{W_R}=n-r$,
i.e. $W_R\in\Phi_{n,r}$. 

The procedure to prove \textbf{(i)}$\implies$\textbf{(iii)} is similar
to the proof \textbf{(i)}$\implies$\textbf{(ii)}.

Next, we prove \textbf{(ii)}$\implies$\textbf{(i)}.
For all $W_R\in\mathbb{S}^n$ such that $0 \preceq W_R \preceq I$, it is true
that
\begin{align}
\label{RC:rc:eq:rank_ineq_W}
\trace{W_R}\leq \rank{W_R}
\end{align}
On the other hand, by using  \emph{Sylvester's Inequality} (see
e.g. \cite[Proposition 2.5.9]{bernstein2009matrix}), we have that
\begin{align}
  \rank{G}+\rank{W_R}\leq n+\rank{GW_R}
\end{align}
Then, by using \eqref{RC:rc:eq:rank_ineq_W}, we have 
\begin{align}
\rank{G}+\trace{W_R}\leq n+\rank{GW_R}
\end{align}
Then by using the fact that $\rank{GW_R}=\rank{0_{m\times n}}=0$ we obtain
\begin{align}
  \rank{G}\leq n-\trace{W_R}
\end{align}
Since $W_R\in\Phi_{n,r}$, we have that $\trace{W_R}=n-r$. Then 
\begin{align}
  \rank{G}\leq r
\end{align}
This completes the proof that \textbf{(ii)}$\implies$\textbf{(i)}. The procedure to prove \textbf{(iii)}$\implies$\textbf{(i)} is similar
to the proof \textbf{(ii)}$\implies$\textbf{(i)}.
\end{proof}

To the best of the authors' knowledge, Theorem
\ref{RC:thm:WG0_gen} is novel. The closest results in rank-constrained optimization, is described in
\cite{markovsky2014recent,markovsky2012low} where the rank-nullity
theorem is used to establish that, for a matrix $G\in\R^{m\times n}$,
\begin{align}
\rank{G}\leq r \iff \mbox{$\exists$ a full row rank matrix
  $U\in\R^{(m-r)\times m}$ such that $UG=0$}
\end{align}  
However, requiring that  $U$ is full row rank  is not
easy. For
example, it may
lead to the necessity of including additional non-convex constraints, such as $UU^\top=I_{m-r}$. 

Another closely related result is described in \cite[\S
4.4]{dattorro2005convex}. The latter result make use of the convex set
$\Phi_{n,r}$, but the formalism is valid only for positive
semidefinite matrices. The above result establishes that for a matrix
$G\in\mathbb{S}^n_+$, 
\begin{align}
\label{eq:rank_dattorro}
\rank{G}\leq r \iff \exists W\in\Phi_{n,r} \mbox{ such that } \trace{WG}=0.
\end{align}
Notice that  Theorem
\ref{RC:thm:WG0_gen} can be seen as a generalisation of
\eqref{eq:rank_dattorro}. 

There exist other rank-constraint representations which impose conditions on the coefficients of the characteristic
polynomial of the matrix, see
\cite{daspremont2003asemidefinite,helmersson2009polynomial}. These representations are valid only for positive semidefinite 
matrices.

Notice that one of the key steps in proving Theorem
\ref{RC:thm:WG0_gen} is the observation that  for all $W\in\mathbb{S}^n$ such that $0
\preceq W \preceq I$, it is true that $\trace{W}\leq\rank{W}$. This
fact is a consequence of a stronger result that says that in the set
of interest, $\{W\in\mathbb{S}^n|0\preceq W \preceq I\}$, 
the trace function is the largest convex function that is less than or
equal to the rank function. This latter result is one of the key
underlying ingredients in the development of the nuclear norm
heuristic \cite{fazel2001arank}.

In the remainder of this section we establish connection between
Theorem \ref{RC:thm:WG0_gen} and other existing results. 
 The following lemma establishes
the relationship between Theorem \ref{RC:thm:WG0_gen} and the
rank-constraint representation in \eqref{eq:rank_dattorro}. 
\begin{lem}
\label{coro:traceWG0}
Let $G\in\mathbb{S}^n_+$ and $W\in\mathbb{S}^n_+$, then  
\begin{align}
\trace{WG}=0 \iff WG=0 
\end{align}  
\end{lem}
\begin{proof}
Since $G$ and $W$ are symmetric and positive semidefinite, then by the \emph{Cholesky
  decomposition}, see e.g. \cite[Fact 8.9.37]{bernstein2009matrix},
there exist matrices $P\in\R^{n\times n}$ and $Q\in\R^{n\times n}$
such that
\begin{align}
G&=PP^\top\\
W&=QQ^\top
\end{align}
We then have that
\begin{align}
\trace{WG}&=\trace{ QQ^\top PP^\top}\\
&=\trace{Q^\top P P^\top Q} \label{RC:eq:tracefrobenius}
\end{align} 
Next, we recall that for $A\in\R^{m \times n}$ the \emph{Frobenius
  norm} is defined by $\|A\|_F=\sqrt{\trace{A^\top A}}$, see
e.g. \cite[page 547]{bernstein2009matrix}. Then, we have 
\begin{align}
\trace{WG}=\trace{Q^\top P P^\top Q}=\|P^\top Q\|_F^2
\end{align}
and from the definition of a norm we have that $\|A\|=0$ if and only
if $A=0$, see e.g. \cite[Definition 9.2.1.]{bernstein2009matrix}. Then
we have that 
\begin{align}
\trace{WG}=\|P^\top Q\|_F^2=0 \implies WG=0
\end{align} 
This concludes the proof for $\trace{WG}=0 \implies WG=0$. The proof for
$WG=0 \implies \trace{WG}=0$ is straightforward. 
\end{proof}

Another particular case of Theorem  \ref{RC:thm:WG0_gen} is the
representation of the $\ell_0$ pseudo norm, denoted as $\|\cdot\|_0$,
which corresponds to the
number of non-zero elements of a vector. The connection is made by
considering a diagonal matrix $G\in\R^{n\times n}$ such that its
diagonal elements are given by a vector $x\in\R^n$,
i.e. $G=\operatorname{diag}\{x\}$ and we have that
$\|x\|_0=\rank{G}$. Then,
Theorem  \ref{RC:thm:WG0_gen}  can be used to prove the following result.   

\begin{coro}
\label{CC:coro:wx0_gen}
Let  $x\in\R^{ n}$, then the following expressions are equivalent
\begin{description}
\item[(i)] $\|x\|_0\leq r$
\item[(ii)] $\exists w\in\{w\in\R^n\vert\;0\leq
  w_i\leq1,i=1,\dots,n;\sum_{i=1}^nw_i=n-r\}$, such that $x_iw_i=0$
  for  $i=1,\dots,n$.
\end{description}
\end{coro}
\begin{proof}
Consider the following definition $G=\operatorname{diag}\{x\}\in\R^{n\times n}$, i.e.
\begin{align}
  G=\begin{bmatrix}x_1&0&0&\cdots&0\\0&x_2&0&\cdots&0\\\vdots
    &\ddots&\ddots&\ddots&\vdots\\ 0&\cdots&0&x_{n-1}&0\\0&\cdots&0&0&x_n\end{bmatrix}
\end{align}
Notice that from construction $\rank{G}=\|x\|_0$. From Theorem \ref{RC:thm:WG0_gen} we have that $\rank{G}\leq r$, if and only
if, there exist a $W\in \{W\in\mathbb{S}^n\vert\;0\preceq W\preceq
I;\trace{W}=n-r\}$ such that $GW=0$. Since $G$ is diagonal, then without
loss of generality, we can assume that $W=\operatorname{diag}\{w\}$.  This can be
easily seen by defining $C=GW$ and considering that $W$ is
symmetric. Note that, since $G$ is diagonal,
$C_{ij}=G_{ii}W_{ij}$ and $C_{ji}=G_{jj}W_{ji}$. If $G_{ii}=G_{jj}=0$
for $i\neq j$
then $W_{ij}=W_{ji}$ can take any value, including zero, and still satisfy 
$C_{ij}=C_{ji}=0$. If $G_{ii}\neq0$ then $W_{ij}=W_{ji}=0$ in order to
satisfy that $C_{ij}=0$.  Finally,  conditions
on $w$ are directly derived from conditions on $W$.
\end{proof}

We note in passing that this representation of $\ell_0$ constraints is related to the results
reported in
\cite{feng2013complementarity,piga2013sdp,daspremont2003asemidefinite},\cite[\S
4.5]{dattorro2005convex}.

\section{Applications in Optimization}

In this section we apply Theorem \ref{RC:thm:WG0_gen} so as to include rank
constraints into optimization problems. In the last decade there has
been  increasing interest on including the rank matrix function into
optimization problems. This is motivated by the introduction of the
development of the nuclear norm heuristic \cite{fazel2001arank}, which
provides a convex relaxation for rank-minimization problems. The
nuclear norm heuristic have been shown to be particularly useful on
high-dimensional optimization problems. However, has been shown in
\cite{markovsky2012howeffective} that there is an inherent loss of
performance on the nuclear norm heuristic.

Theorem \ref{RC:thm:WG0_gen}  can be applied to rank-constrained
optimization problems by simply replacing the rank contraint by one of
the equivalent representations, as follows        
\begin{align*}
\begin{aligned}
 &\mathcal{P}_{rco}:&\min_{\theta\in\mathbb{R}^p}
\;\;&f(\theta) \\
 &&\text{s.t.}
 \;\;&\theta\in\Omega\\
&&& \rank{G(\theta)}\leq r
\end{aligned}\equiv
\begin{aligned}
 &\mathcal{P}_{rco equiv}:&\min_{\theta\in\mathbb{R}^p}\min_{W\in\mathbb{S}^n}
\;\;&f(\theta) \\
 &&\text{s.t.}
 \;\;&\theta\in\Omega\\
&&& G(\theta) W=0_{m\times n}\\
&&&W\in\Phi_{n,r}
\end{aligned}
\end{align*}

On the other hand,  Theorem \ref{RC:thm:WG0_gen}  can also be applied
to rank-minimization problems by using the epigraph representation
\cite{grant2008graph}, as follows

\begin{align*}
\begin{aligned}
 &\mathcal{P}_{rm}:&\min_{\theta\in\mathbb{R}^p}
\;\;&r \\
 &&\text{s.t.}
 \;\;&\theta\in\Omega\\
&&& \rank{G(\theta)}\leq r
\end{aligned}\equiv
\begin{aligned}
 &\mathcal{P}_{rm equiv}:&\min_{\theta\in\mathbb{R}^p}\min_{W\in\mathbb{S}^n}
\;\;&n-\trace{W} \\
 &&\text{s.t.}
 \;\;&\theta\in\Omega\\
&&& G(\theta) W=0_{m\times n}\\
&&&0\preceq W\preceq I_n
\end{aligned}
\end{align*}

These ideas has been applied by the current authors to rank-constrained optimization problems.
For example, in \cite{aguilera2014quadratic} Corollary
\ref{CC:coro:wx0_gen} has been used to impose $\ell_0$ contraints into
a Model Predictive Control problem. In \cite{delgado2014arank} the
problem of Factor Analysis is considered. In the latter work, Theorem
\ref{RC:thm:WG0_gen} has been used to relax the restrictive assumption that
the noise sequences should be uncorrelated. The  equivalence for the rank-minimization
problem $\mathcal{P}_{rm}$ to the problem $\mathcal{P}_{rm equiv}$ can be seen as a generalisation  to 
non-square matrices  of the
results presented in\cite{daspremont2003asemidefinite}.

 \bibliographystyle{apalike}
 \bibliography{rank_equiv}

\end{document}